\documentclass{article}
\usepackage{alltt}
\usepackage{amsmath}
\usepackage{amssymb}
\usepackage{amsthm}
\usepackage{bm}
\usepackage{breqn}
\usepackage{dsfont}
\usepackage{enumerate}
\usepackage{fancyref}
\usepackage{geometry}
\usepackage{graphics}
\usepackage{graphicx}
\usepackage[mathlines]{lineno}
\usepackage{mathtools}
\usepackage{standalone}
\usepackage{tikz}
\usetikzlibrary{calc}
\usepackage{pgffor}

\usepackage{hyperref}

\newtheorem{theorem}{Theorem}
\newtheorem{definition}[theorem]{Definition}

\newcommand{\card}[1]{\left\lvert#1\right\rvert}
\newcommand{\set}[1]{\left\{#1\right\}}
\newcommand{\ceil}[1]{\left\lceil#1\right\rceil}
\newcommand{\floor}[1]{\left\lfloor#1\right\rfloor}
\newcommand{\pare}[1]{\left(#1\right)}

\mathtoolsset{showonlyrefs}

\title{On B-Colorings in Planar Graphs}
\author{Anthony Vuolo\footnote{Department of Mathematics, Iowa State University, Ames, IA 50014, USA. Email: \href{mailto:vuolo@iastate.edu}{vuolo@iastate.edu}}}
\date{\today}

\begin{document}
\maketitle
\begin{abstract}
Gy\'arf\'as and S\'ark\"ozy [Studia Sci. Math. Hungar., 2023] defined a B-coloring of a graph to be a proper coloring of the edge set in which any $C_4$ is totally multicolored. 
Let $q_B(G)$ denote the minimum number of colors sufficient for a B-coloring of a graph $G$. 
In this paper, we prove that any planar graph $G$ with $\Delta=\Delta(G)$ and $\Delta_2=\Delta_2(G)$ has $q_B(G)\leq\Delta+\max\{\Delta_2,38\}$, refining a bound by Kong, Wang, and Zheng [J. Graph Theory, 2026].
\end{abstract}

\section{Background}
A coloring of the edge set of a graph is said to be proper if no two incident edges have the same color. 
This can be rephrased as every color class is a matching but can also be rephrased as every copy of $K_{1,2}$ is totally multicolored.
The edge-chromatic number (or chromatic index) of a graph $G$, denoted $\chi'(G)$, is the minimum number of colors sufficient to properly color the edge set of $G$.

In 1964, Vizing~\cite{Vizing_1964} proved that $\Delta(G)\leq\chi'(G)\leq\Delta(G)+1$ for any simple graph $G$ and all graphs in this paper are simple. 
Graphs $G$ for which $\chi'(G)=\Delta(G)$ are said to be in \emph{Class 1} and graphs $G$ for which $\chi'(G)=\Delta(G)+1$ are said to be in \emph{Class 2}.
Moreover, in 1965, Vizing~\cite{Vizing_1965} also showed every planar graph $G$ with $\Delta(G)\geq 8$ is in Class 1. 
In that same paper, Vizing conjectured that no Class 2 planar graphs exist with $\Delta(G)\in\{6,7\}$. Furthermore, he gave examples of Class 2 planar graphs with maximum degree $\Delta(G)=k$ for $k\in\{2,3,4,5\}$: 
\begin{itemize}
    \item $k=2$: Any odd cycle.
    \item $k=3$: The tetrahedral graph with one edge subdivided.
    \item $k=4$: The octahedral graph with one edge subdivided.
    \item $k=5$: The icosahedral graph with one edge subdivided.
\end{itemize}

In 2001, Sanders and Zhao~\cite{SZ_2001} verified Vizing's conjecture in the case where $\Delta(G)=7$, showing that no Class 2 graphs exist with $\Delta(G)\geq 7$. 
It is still unknown whether there is a Class 2 planar graph with $\Delta(G)=6$.

The literature is replete with edge colorings that have more restrictions than simply being proper.
A \emph{strong edge coloring} of a graph is defined to be a proper edge coloring in which each color class induces a matching. 
This can be rephrased as every copy of $K_{1,2}=P_3$ and every copy of $P_4$ is totally multicolored, where $P_k$ denotes the path on $k$ vertices. 
The \emph{strong edge coloring number} of a graph $G$ is the minimum number of colors sufficient for a strong edge coloring and is denoted $\chi_s'(G)$.
Brown, Erd\H{o}s, and S\'os \cite{BES_1973} conjectured that for any graph $G$, it holds that
\begin{equation}
    \chi_s'(G)=\left\{\begin{array}{ll}
        \frac{5}{4}(\Delta(G))^2, &\text{ if }\Delta(G)\text{ is even;}\\
        \frac{1}{4}\pare{5(\Delta(G))^2-2\Delta(G)+1}, &\text{ if }\Delta(G)\text{ is odd.}
    \end{array}\right.
\end{equation}
Molloy and Reed \cite{MR_1997} proved that $\chi_s'(G)\leq 1.998(\Delta(G))^2$ for any graph $G$. 

Faudree, Schelp, Gy\'arf\'as, and Tuza~\cite{FSGT_1990} proved that $\chi_s'(G)\leq 4\Delta+4$ for any planar graph $G$ with maximum degree $\Delta$. 
Furthermore, for all $\Delta\geq 2$, they give a construction of a planar graph that meets this upper bound.
Other results related to strong edge colorings are given in~\cite{BI_strong_2013,MS_2017,Deniz_2024,Wang_2025}.

In this paper, we consider the notion of a \emph{B-coloring}, which was defined by Gy\'arf\'as and S\'ark\"ozy~\cite{GS_2023}.
A B-coloring of the edges of a graph is one in which every copy of $K_{1,2}$ and every copy of $K_{2,2}=C_4$ is totally multicolored. 
The minimum number of colors sufficient for a B-coloring of a graph $G$ is denoted $q_B(G)$ and is called the \emph{B-coloring number} of $G$.
In the same paper, Gy\'arf\'as and S\'ark\"ozy also define \emph{A-colorings}, which are proper edge colorings in which no two color classes can induce a path or cycle on four edges. 
The corresponding \emph{A-coloring number} is denoted $q_A(G)$.
Gy\'arf\'as and S\'ark\"ozy provide a connection between $q_A(G),q_B(G)$ and the (7,4)-conjecture first posed by Erd\H{o}s in \cite{BES_1973}, by showing that a certain bound on the A-coloring or B-coloring numbers on any balanced bipartite graph would prove the conjecture.
For any graph $G$, there are trivial bounds of $\Delta(G)\leq\chi'(G)\leq q_B(G)\leq \bigl(\Delta(G)\bigr)^2$. 

In the case where $G$ is planar, Gy\'arf\'as, Martin, Ruszink\'o, and S\'ark\"ozy~\cite{GMRS_2024} proved that $q_B(G)\leq 2\Delta(G)+8$ for planar $G$ and conjectured that $q_B(G)\leq 2\Delta(G)$ when $\Delta(G)$ is sufficiently large.
The graph $K_{2,d}$ gives an example of a planar graph for which $q_B\left(K_{2,d}\right)=2\Delta\left(K_{2,d}\right)=2d$, so this bound is best possible. 
Furthermore, there are examples of graphs for which the requirement of $\Delta(G)$ being sufficiently large is necessary. 
The complete tripartite graph $K_{2,2,2}$ is planar and $\Delta\bigl(K_{2,2,2}\bigr)=4$, but $q_B\bigl(K_{2,2,2}\bigr)=12$. 

In a later manuscript, Martin, Ruszink\'o, and S\'ark\"ozy~\cite{MRS_2026} established that $q_B(G)\leq \max\{2\Delta(G)+4,22\}$.
Finally, the conjecture from~\cite{GMRS_2024} was verified by Kong, Wang, and Zheng~\cite{KWZ_2026}, who proved the following:
\begin{theorem}[Kong, Wang, Zheng~\cite{KWZ_2026}]
    If $G$ is a planar graph with maximum degree $\Delta\geq 38$, then $q_B(G)\leq 2\Delta$.
    Furthermore, $q_B(G)\leq 2\Delta+6$ and, if $\Delta\geq 12$, then $q_B(G)\leq 2\Delta+4$. \label{thm:KWZ_2026}
\end{theorem}

In this paper, we use the parameter $\Delta_2(G)$, which denotes the maximum co-degree of the graph. 
That is, $\Delta_2(G)$ is the maximum $t$ such that $G$ has a subgraph isomorphic to $K_{2,t}$. 
Note that, for any graph $G$, 
\begin{align}
    \max\bigl\{\Delta(G),2\Delta_2(G)\bigr\}\leq\max\bigl\{\chi'(G),2\Delta_2(G)\bigr\}\leq q_B(G)\leq\max\bigl\{2\Delta(G),76\bigr\} . \label{eq:Delta2:LB}
\end{align}
Observe that if $\Delta\geq 3$, then the book graph $G=K_{1,1,\Delta-1}$ is a planar graph such that $\chi'(G)=\Delta$, $2\Delta_2(G)=2\Delta-2$ but $q_B(G)=2\Delta-1$ because all edges must have a different color and so it strictly exceeds the lower bound in~\eqref{eq:Delta2:LB}.

We will refine the bound in Theorem~\ref{thm:KWZ_2026} by including $\Delta_2(G)$ in the upper bound. 
\begin{theorem}
\label{thm:dd2c}
    If $G$ is a planar graph with $\Delta=\Delta(G)$ and $\Delta_2=\Delta_2(G)$, then $q_B(G)\leq\Delta+\max\set{\Delta_2,38}.$
\end{theorem}

This result improves upon the main result in Theorem~\ref{thm:KWZ_2026} of $q_B(G)\leq 2\Delta$ for $\Delta\geq 38$ and strictly so except in the case where $\Delta=38$. 

The rest of the paper is organized as follows: 
In Section~\ref{sec:term}, we present the terminology that is used throughout the paper. 
In Section~\ref{sec:proof}, we present the proof of Theorem~\ref{thm:dd2c}.
In Section~\ref{sec:conc}, we offer some concluding remarks and open questions, followed by acknowledgements and references.

\section{Terminology}
\label{sec:term}
For vertex $u$, we denote $N(u)$ to be the neighborhood of $u$ and $\deg(u)=|N(u)|$.
For distinct vertices $u,w$, we denote $\deg(u,w)=\bigl|N(u)\cap N(w)\bigr|$.

We will borrow language and terminology from a key paper by Borodin, Broersma, Glebov, and van den Heuvel~\cite{BBGH_2002}. 

Recall that a \emph{plane graph} is a drawing of a graph in the plane without edge-crossings and a \emph{planar graph} is a graph that admits such a drawing.

A cycle in a plane graph is said to be \emph{separating} if there exists a vertex that lies inside of the cycle. 
We wish to emphasize that whether or not a cycle is separating can depend on the particular drawing that is chosen. 
\begin{definition}
    \label{def:bunch}
    For distinct vertices $x$ and $y$ and integer $m\geq 3$, a \emph{bunch} $B(x,y;m)$ in a plane graph $G$ is an induced subgraph in which the $x$ and $y$ are designated as \emph{poles} and there are paths $P_1,P_2,\ldots,P_m$ such that 
\begin{enumerate}[(a)]
    \item each $P_i$ has length 1 or 2 and joins $x$ and $y$,
    \item the cycle formed by $P_i$ and $P_{i+1}$ is not separating in $G$, for each $i\in\{1,\ldots,m-1\}$, and
    \item $m$ is maximal, in that no path $P_0$ or $P_{m+1}$ can be added to the bunch while preserving the first two properties.
\end{enumerate}
If a path $P_i$ is of the form $xz_iy$, we call vertex $z_i$ a \emph{brother}.
\end{definition}

We note that, in a bunch, at most one of the paths $P_i$, $i\in\{1,\ldots,m\}$ can have length 1, otherwise the graph is no longer simple. 
Figure \ref{fig:bunch_example} shows a bunch without an edge between $x$ and $y$ and a bunch with an edge between $x$ and $y$. 
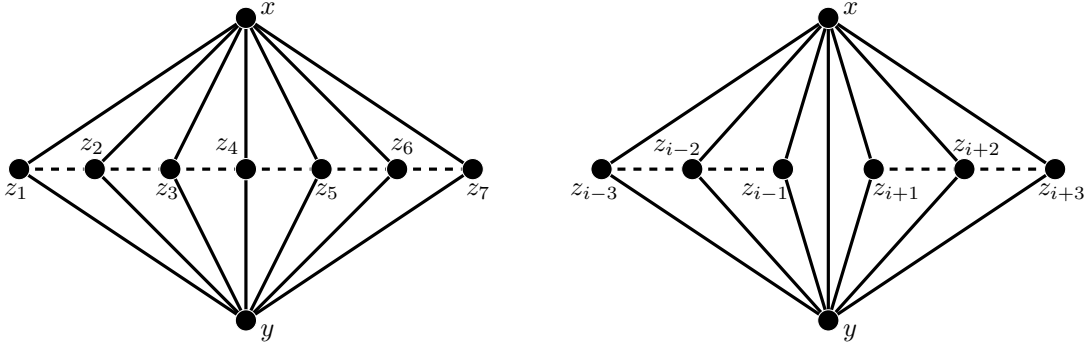
\begin{figure}
    \centering 
    ~\hfill
    \begin{minipage}{0.4\textwidth}
        \centering
        \begin{tikzpicture}

\tikzset{
line/.style={-,very thick},
blackdotl/.style={color=black,dashed}
}
\def\lenbun{7}
\def\hei{2}
\def\wid{1.0}
\def\radp{4pt}
\def\offset{0.5}
\useasboundingbox (0-\offset,0-\offset) rectangle ({(\lenbun-1)*\wid+\offset},2*\hei+\offset);
\path ({(\lenbun-1)*0.5*\wid},0) coordinate (y);
\path ({(\lenbun-1)*0.5*\wid},{2*\hei}) coordinate (x);
    \foreach \i in {1,...,\lenbun} {
        \path ({(\i-1)*\wid},\hei) coordinate (z\i);
    };
    \foreach \i [evaluate=\i as \inext using {int(\i+1)}] in {1,...,\lenbun} {
        \draw[line] (x) -- (z\i) -- (y);
        \ifnum \i<\lenbun
            \draw[line,blackdotl] (z\i) -- (z\inext);
        \fi
    };
    \foreach \i in {1,...,\lenbun} {
        \draw[thin, white, fill=black] (z\i) circle[radius=\radp];
    };
    \draw[thin, white, fill=black] (x) circle[radius=\radp];
    \draw[thin, white, fill=black] (y) circle[radius=\radp];
    \node[above right,xshift=2pt,yshift=-2pt] at (x) {$x$};
    \node[below right,xshift=2pt,yshift=2pt] at (y) {$y$};
    \node[below left,xshift=7pt,yshift=-2pt] at (z1) {$z_1$};
    \node[above left,xshift=7pt,yshift=+2pt] at (z2) {$z_2$};
    \node[below left,xshift=7pt,yshift=-2pt] at (z3) {$z_3$};
    \node[above, xshift=-7pt,yshift=+2pt] at (z4) {$z_4$};
    \node[below right,xshift=-6pt,yshift=-2pt] at (z5) {$z_5$};
    \node[above right,xshift=-6pt,yshift=+2pt] at (z6) {$z_6$};
    \node[below right,xshift=-6pt,yshift=-2pt] at (z7) {$z_7$};
\end{tikzpicture}
    \end{minipage}~
    \hfill
    \begin{minipage}{0.4\textwidth}
        \centering
        \begin{tikzpicture}

\tikzset{
line/.style={-,very thick},
blackdotl/.style={color=black,dashed}
}
\def\lenbun{6}
\def\hei{2}
\def\wid{1.2}
\def\radp{4pt}
\def\offset{0.5}
\useasboundingbox (0-\offset,0-\offset) rectangle ({(\lenbun-1)*\wid+\offset},2*\hei+\offset);
\path ({(\lenbun-1)*0.5*\wid},0) coordinate (y);
\path ({(\lenbun-1)*0.5*\wid},{2*\hei}) coordinate (x);
    \foreach \i in {1,...,\lenbun} {
        \path ({(\i-1)*\wid},\hei) coordinate (z\i);
    };
    \foreach \i [evaluate=\i as \inext using {int(\i+1)}] in {1,...,\lenbun} {
        \draw[line] (x) -- (z\i) -- (y);
        \ifnum \i<3
            \draw[line,blackdotl] (z\i) -- (z\inext);
        \fi
        \ifnum \i>3
            \ifnum \i<\lenbun
                \draw[line,blackdotl] (z\i) -- (z\inext);
            \fi
        \fi
    };
    \draw[line] (x) -- (y);
    \foreach \i in {1,...,\lenbun} {
        \draw[thin, white, fill=black] (z\i) circle[radius=\radp];
    };
    \draw[thin, white, fill=black] (x) circle[radius=\radp];
    \draw[thin, white, fill=black] (y) circle[radius=\radp];
    \node[above right,xshift=2pt,yshift=-2pt] at (x) {$x$};
    \node[below right,xshift=2pt,yshift=2pt] at (y) {$y$};
    \node[below left,xshift=6pt,yshift=-2pt] at (z3) {$z_{i-1}$};
    \node[above left,xshift=7pt,yshift=1pt] at (z2) {$z_{i-2}$};
    \node[below left,xshift=10pt,yshift=-2pt] at (z1) {$z_{i-3}$};
    \node[below right,xshift=-4pt,yshift=-2pt] at (z4) {$z_{i+1}$};
    \node[above right,xshift=-8pt,yshift=1pt] at (z5) {$z_{i+2}$};
    \node[below right,xshift=-10pt,yshift=-2pt] at (z6) {$z_{i+3}$};

\end{tikzpicture}
    \end{minipage}~
    \hfill~
    \caption{From Definition~\ref{def:bunch}, a bunch with no edge $(x,y)$ and $m=7$ (left) and a bunch with edge $(x,y)$ (right). Each dotted edge may or may not exist in the graph}
    \label{fig:bunch_example}
\end{figure}

\begin{definition}
    A \emph{precomplete star centered at} vertex $u$ is the star formed by $u$, together with $\deg(u)-1$ of its neighbors, $v_1,\ldots,v_k$, where $k=\deg(u)-1$. 
    The \emph{weight} of this star is defined to be $\sum_{i=1}^k\deg(v_i)$. 
    The vertex $u$ is a \emph{minor vertex} if $\deg(u)\leq 5$.
\end{definition}
Note that Euler's theorem establishes $e(G)\leq 3v(G)-6$ for every planar graph $G$ on at least three vertices, which implies that every planar graph has a minor vertex. 

Theorem~\ref{thm:bunchminor38} is a key ingredient for the proof of Theorem~\ref{thm:KWZ_2026} in~\cite{KWZ_2026} and is also key to the proof of Theorem~\ref{thm:dd2c}.
\begin{theorem}[Borodin, Broersma, Glebov, van den Heuvel~\cite{BBGH_2002}]
\label{thm:bunchminor38}
    Any planar graph $G$ contains one of the following substructures:
    \begin{enumerate}[(1)]
        \item A precomplete star with weight at most $38$ centered at a minor vertex. (This is vacuously satisfied if $G$ has a leaf, i.e. if $\delta(G)=1$.)
        \item A bunch $B(x,y;m)$ where $\deg(x)\geq26$ and $m\geq\deg(x)/5$.
    \end{enumerate}
\end{theorem}

\section{Proof of Theorem~\ref{thm:dd2c}}
\label{sec:proof}
Fix $\Delta$ and $\Delta_2$. 
For a graph $G$ with $\Delta(G)\leq\Delta$ and $\Delta_2(G)\leq\Delta_2$, we call a B-coloring of $G$ \textit{small} if it uses at most $\Delta+\max\set{\Delta_2,38}$ colors. 
\begin{proof}[Proof of Theorem \ref{thm:dd2c}]
    We proceed by induction on $\card{V(G)}$. 
    Note that by deleting a vertex $v$ from $G$, the maximum degree does not increase, nor does the maximum co-degree. 
    Furthermore, no conflicts between edges that are not incident to $v$ are removed by the deletion of $v$.
    Therefore, we may B-color $G-v$ and, if enough colors remain, color the edges incident to $v$. 
    
    The base case is $\max\bigl\{1,3\card{V(G)}-6\bigr\}\leq\Delta+\max\{\Delta_2,38\}$ (which is satisfied if $n\leq 15$). 

    We have two main cases, defined according to Theorem~\ref{thm:bunchminor38}:\\
    
    \noindent\textbf{Case 1.} $G$ has a precomplete star of weight at most 38 at a minor vertex $v$. \\

    Write $H=G-v$. By induction, $H$ has a B-coloring $\gamma$, which is small. 
    When considered as a coloring of $G$, there are still no two edges with the same color that conflict in $G$. 
    Now we extend $\gamma$ to a B-coloring of $G$.
    Let $k=\deg(v)-1$ and let $v_1,\dots,v_{k+1}$ be the neighbors of $v$, and write $e_i=vv_i$ for $i=1,\ldots,k+1$. 

    For any $e_i$, a conflicting edge must be incident to $v_j$ for some $j\in\{1,\ldots,k+1\}$, otherwise the edges would not be in the same $K_{1,2}$ or the same $K_{2,2}$. 
    Hence the number of edges conflicting with $e_i$, in $H$, is at most 
    \begin{equation}
        \sum_{i=1}^{k+1}\bigl(\deg(v_i)-1\bigr)\leq\Delta+38-(k+1).
    \end{equation}
    Since there are $\Delta+\max\{\Delta_2,38\}\geq \Delta+38$ colors in the palette, there are at least $k+1$ colors available for $e_1,\ldots,e_{k+1}$
    So the edges $e_1,\ldots,e_{k+1}$ can be colored greedily from those remaining colors. 
    This concludes Case 1. \\

    \noindent\textbf{Case 2.} $G$ has a bunch $B(x,y;m)$ where $\deg(x)\geq 26$ and $m\geq\deg(x)/5$. \\

    Observe that $m\geq\lceil\deg(x)/5\rceil\geq\lceil 26/5\rceil\geq 6$.
    If $xy\not\in E(G)$, then write $z_1,\dots,z_m$ for the brothers.
    If $xy\in E(G)$, then write $z_1,\dots,z_{i-1},z_{i+1},\ldots,z_m$ for the brothers, where $P_i$ is the path of length 1.\\

    \noindent\textbf{Case 2a.} $xy\not\in E(G)$ and $m\geq 7$. \\
        
    Let $v=z_4$ and write $H=G-v$. 
    Note that $\deg(v)\leq 4$ and the edges incident to $v$ are $xv$, $yv$, $vz_3$ if it exists, and $vz_5$ if it exists. 
    Let $\gamma$ be a B-coloring of $H$ that is small. 
    Observe that if $vz_3$ exists, then the edges with which it conflicts is a subset of $\bigl\{xz_i,yz_i: i\in\{2,3,4,5\}\bigr\}$ as well as $z_2z_3$ and $vz_5$, if they exist. 
    Similarly for $vz_5$.
    Therefore, because the palette has size at least 14, then if there is a B-coloring of $H$ which is extended to color $xv$ and $yv$ such that no conflict is created, then $vz_3$ and $vz_5$ can be colored greedily if they exist.

    However, the edges that conflict with $xv$ in graph $H$ are any edge incident to $x$ plus edges of the form $yw$ such that $xw\in E(H)$ (these may not be part of the bunch) plus $z_2z_3$ and $z_5z_6$, if they exist.
    This gives at most $(\deg(x)-1)+(\deg(x,y)-1)+2$ conflict edges, which is at most $\Delta+\Delta_2$ if both $z_2z_3$ and $z_5z_6$ exist. 
    The same bound holds for the number of edges that conflict with $yv$.
    If $z_2z_3$ exists, then we will recolor it to have a color that exists among the edges that conflict with $xv$ and among the edges that conflict with $yv$. 
    We will do the same for $z_5z_6$ if it exists. 
    Then $xv$ and $yv$ conflict with at most $\Delta+\Delta_2-2$ colors and can be colored greedily. 
    Afterwards, $vz_3$ and $vz_5$ can be colored greedily to complete the B-coloring.

    Now we show the recoloring.

    If the edge $z_2z_3$ exists, then in $H$, it can only conflict with edges $xz_i$ for $i\in\{1,2,3,4,5\}$, $yz_i$ for $i\in\{1,2,3,4,5\}$, and $z_1z_2$ if it exists. 
    Of these, they all conflict with each of $xz_6$ and $yz_6$, except $z_1z_2$, which does not conflict with either $xz_6$ or $yz_6$. 
    Hence, only $z_1z_2$ may share a color with either $xz_6$ or $yz_6$. 
    Thus, if $z_2z_3$ exists, we may recolor $z_2z_3$ with whichever of the colors on $xz_6$ or $yz_6$ is not used on $z_1z_2$, if it exists; otherwise, we may use either color.
    Thus, $z_2z_3$ receives the same color as either $xz_6$ or $yz_6$.
    
    By parallel logic, if $z_5z_6$ exists, we can recolor $z_5z_6$ with one of the colors on $xz_2$ or $yz_2$.
    Consequently, the number of colors that conflict with $xv$ and the number of colors that conflict with $yv$ are at most $\Delta+\Delta_2-2$ and we may finish the B-coloring with the greedy procedure described above. 
    
    This concludes Case 2a. \\

    \noindent\textbf{Case 2b.} $xy\not\in E(G)$ and $m=6$. \\
        
    By Theorem~\ref{thm:bunchminor38}, $\deg(x)\leq 5m=30$. 
    Let $v=z_4$ and write $H=G-v$. 
    First, note that $yv$ conflicts, in $H$, with at most $\Delta-1$ edges incident to $y$, at most $\deg(x)-1\leq 29$ edges incident to $x$, and $z_2z_3$ and $z_5z_6$ if they exist. 
    This leaves at least 7 available colors for $yv$.
    
    Next, note that $xv$ conflicts, in $H$, with at most $\deg(x)-1\leq 29$ edges incident to $x$, $\min\bigl\{\deg(x)-1,\Delta_2-1\bigr\}$ edges incident to $y$, and possibly $z_2z_3$,$z_5z_6$ as before.
    This leaves at least 7 available colors for $xv$.

    If $z_3v$ exists, then it can conflict with $\set{xz_i, yz_i : i\in\set{2,3,5,6}}$ plus $z_2z_3$ if it exists. 
    If $z_5v$ exists, then it can conflict with $\set{xz_i,yz_i:i\in\set{2,3,5,6}}$ plus $z_5z_6$ if it exists. 
    This leaves at least 29 available colors for $z_3v$ and for $z_5v$.
    
    Consequently, we may finish the B-coloring greedily. 

    This concludes Case 2b. \\

    \noindent\textbf{Case 2c.} $xy\in E(G)$. \\
            
    By the definition of $i$, if $2\leq i\leq m-1$, both $xyz_{i-1}$ and $xyz_{i+1}$ are triangles. 
    By symmetry, we may assume that $i\geq \floor{m/2}+1$, thus $i\geq 4$. 
    
    Let $v=z_{i-1}$ and write $H=G-v$. 
    Let $\gamma$ be a B-coloring of $H$ which is small. 
    We have that $xv$ conflicts with at most $\Delta-1$ edges incident to $x$, at most $\Delta_2-2$ edges incident to $y$ (excluding $xy$ which is already counted), and $z_{i-2}z_{i-1}$ if it exists. 
    The total number of edges in $H$ that conflict with $xv$ is at most $\Delta+\Delta_2-2$, which leaves at least 2 colors available for $xv$.
    The same bound holds for the number of edges in $H$ that conflict with $yv$.
 
    If $z_{i-2}v$ exists, then it conflicts with $xy$, $\set{xz_j,yz_j: j\in\set{i-3,i-2}}$ and $z_{i-3}z_{i-2}$ if it exists. 
    This is at most 6 edges, which leaves at least 32 colors available for $z_{i-2}v$.

    Consequently, we may finish the B-coloring greedily by first coloring $xv$ and $yv$ and then coloring $z_{i-2}v$.  

    This concludes Case 2c. \\

    Thus, in all cases we have a B-coloring of $G$ which is small.
\end{proof}

\section{Conclusions and future work}
\label{sec:conc}
\subsection{Improving Results on Planar Graphs}
Our theorem motivates two questions regarding B-coloring planar graphs:
\begin{enumerate}[(1)]
    \item What is the smallest positive constant $C$ such that for any planar graph $G$, $q_B(G)\leq\Delta+\max\set{\Delta_2,C}$?

    Our theorem gives $C\leq 38$. 
    The graph $G=K_{2,2,2}$ gives $\Delta(G)=\Delta_2(G)=4$ and $q_B(G)=12$, so $C\geq 8$ is necessary. 
    This motivates the second question: 
    
    \item Are there elementary graph parameters that can provide good lower and upper bounds on $q_B(G)$?

    It is clear that, for any planar graph $G$, $q_B(G)\geq\max\set{\Delta,2\Delta_2}$. Is it true that there is a constant $C'$ such that for any planar graph $G$, $q_B(G)\leq\max\set{\Delta,2\Delta_2,C'}$?
\end{enumerate}

\subsection{Outerplanar Graphs}
Gy\'arf\'as, Martin, Ruszink\'o, and S\'ark\"ozy~\cite{GMRS_2024} pose a similar conjecture for outerplanar graphs $G$: is $q_B(G)=\Delta$ when $\Delta$ is sufficiently large? In other words, is there a constant $C$ such that $q_B(G)\leq\max\set{\Delta,C}$ for all outerplanar $G$? 
This was proven true by Kong, Wang, and Zheng~\cite{KWZ_2026}, who showed that $q_B(G)=\Delta$ when $\Delta\geq 7$.
Martin, Ruszink\'o, and S\'ark\"ozy~\cite{MRS_2026} proved that for any outerplanar graph $G$, $q_B(G)\leq\max\set{\Delta,6}$
The graph $G_1=K_{1,1,2}$ is outerplanar with $\Delta(G_1)=3$ and $q_B(G_1)=5$, and the graph $G_2=K_5-P_4$ is outerplanar with $\Delta(G_2)=4$ and $q_B(G_2)=5$. 
Hence, $C\geq 5$. 
It remains to show whether $C=5$ or $C=6$.

\subsection{General Graphs}
We also ask the question of what parameters can be used to provide bounds on $q_B(G)$ for general graphs $G$. 
Note that in a complete multipartite graph, every pair of edges is either incident or in the same $C_4$. 
Hence, an immediate lower bound for $q_B(G)$ is the maximum number of edges in a complete multipartite subgraph. 
We will denote this parameter for graph $G$ by $t(G)$ (for ``Tur\'an graph''). 
It is not clear what an upper bound or better lower bound would be in terms of $t(G)$. 
Gy\'arf\'as and S\'ark\"ozy~\cite{GS_2023} observed that if $G$ is a simple grpah with maximum degree $\Delta$, then $q_B(G)\leq \Delta^2,$
and this immediately implies $q_B(G)\leq (t(G))^2$ since a degree-$\Delta$ vertex is a complete multipartite graph.

For a better lower bound, consider the graph $C_5(r)$, which is isomorphic to the balanced blow-up of $C_5$ with parts of size $r$. 
This gives $t\bigl(C_5(r)\bigr)=2r^2$ by taking the blowup on one vertex in $C_5$ as a part of size $r$ and the blowups of its neighbors as a part of size $2r$.
If any three edges in $C_5(r)$ share a color, then two of those edges are in blowups of either the same edge or incident original edges in $C_5$. 
Hence, those two edges are either incident or in the same $C_4$, a contradiction. 
As a result, any color class in $C_5(r)$ has at most two edges.
This gives
$$q_B(C_5(r))\geq \ceil{\frac{5}{2}r^2}=\ceil{\frac{5}{4}t\bigl(C_5(r)\bigr)}.$$
This bound is tight; we may choose color sets $\gamma_i,i\in\set{1,2,3,4,5}$ with size $\ceil{r^2/2}$ if $i$ is odd or $\floor{r^2/2}$ if $i$ is even. The blowups of the edges can be colored using the two sets $\gamma_1\gamma_2,\gamma_3\gamma_4,\gamma_1\gamma_5,\gamma_2\gamma_3,\gamma_4\gamma_5$ in cyclic order.
Therefore we cannot obtain a Vizing-type theorem in which $t(G)$ is both a lower bound and within a constant of an upper bound. 
However, that does not preclude another elementary parameter from providing a Vizing-type theorem for B-colorings of general simple graphs. 

\section*{Acknowledgements}
The author would like to thank Ryan Martin for mentoring and sponsoring this research. 
The author would also like to thank Yuping Gao for alerting him to the publication of \cite{KWZ_2026}, giving inspiration for this paper.

The example of $K_{1,1,\Delta-1}$ which established that $q_B(G)>\max\set{\chi'(G),2\Delta_2(G)}$ was found with the assistance of M365 Copilot based on the GPT-5 chat model, accessed 28 July 2026. 

\bibliographystyle{plain}
\bibliography{biblio}

@incollection {FSGT_1990,
    AUTHOR = {Faudree, R. J. and Schelp, R. H. and Gy\'arf\'as, A. and Tuza,
              Zs.},
     TITLE = {The strong chromatic index of graphs},
      NOTE = {Twelfth British Combinatorial Conference (Norwich, 1989)},
   JOURNAL = {Ars Combin.},
  FJOURNAL = {Ars Combinatoria. A Canadian Journal of Combinatorics},
    VOLUME = {29},
      YEAR = {1990},
     PAGES = {205--211},
      ISSN = {0381-7032,2817-5204},
   MRCLASS = {05C15},
  MRNUMBER = {1412876},
}

@article{KWZ_2026,
    AUTHOR = {Kong, Jiangxu and Wang, Yiqiao and Zheng, Mengmeng},
     TITLE = {B-Coloring of Planar Graphs},
   JOURNAL = {Journal of Graph Theory},
      NOTE = {Published online, 12pp.},
       DOI = {https://doi.org/10.1002/jgt.70067},
       URL = {https://onlinelibrary.wiley.com/doi/abs/10.1002/jgt.70067},
    EPRINT = {https://onlinelibrary.wiley.com/doi/pdf/10.1002/jgt.70067},
      YEAR = {2026}
}

@techreport{BBGH_2002,
    AUTHOR = {Borodin, O. V. and Broersma, H. J. and Glebov, A. N. and {van den Heuvel}, J.},
     TITLE = {Stars and bunches in planar graphs. Part {II}: General planar graphs and colourings},
INSTITUTION = {London School of Economics},
      TYPE = {CDAM Research Report},
    NUMBER = {CDAM-2002-05},
      YEAR = {2002}
}

@article {Vizing_1964,
    AUTHOR = {Vizing, V. G.},
     TITLE = {On an estimate of the chromatic class of a {$p$}-graph},
   JOURNAL = {Diskret. Analiz},
  FJOURNAL = {Akademiya Nauk SSSR. Sibirskoe Otdelenie. Institut Matematiki.
              Diskretny\u i\ Analiz. Sbornik Trudov},
      YEAR = {1964},
    NUMBER = {3},
     PAGES = {25--30},
   MRCLASS = {05.55},
  MRNUMBER = {180505},
MRREVIEWER = {J.\ Bos\'ak},
}

@article{Vizing_1965,   
    AUTHOR = {Vizing, V. G.},
     TITLE = {Critical graphs with given chromatic class},
   JOURNAL = {Diskret. Analiz},
  FJOURNAL = {Akademiya Nauk SSSR. Sibirskoe Otdelenie. Institut Matematiki.
              Diskretny\u i\ Analiz. Sbornik Trudov},
      YEAR = {1965},
    NUMBER = {5},
     PAGES = {9--17},
   MRCLASS = {05.55},
  MRNUMBER = {200202},
MRREVIEWER = {J.\ Bos\'ak},
}

@article{SZ_2001,
    AUTHOR = {Sanders, Daniel P. and Zhao, Yue},
     TITLE = {Planar graphs of maximum degree seven are class {I}},
   JOURNAL = {J. Combin. Theory Ser. B},
  FJOURNAL = {Journal of Combinatorial Theory. Series B},
    VOLUME = {83},
      YEAR = {2001},
    NUMBER = {2},
     PAGES = {201--212},
      ISSN = {0095-8956,1096-0902},
   MRCLASS = {05C15 (05C35)},
  MRNUMBER = {1866396},
MRREVIEWER = {H.\ L.\ Abbott},
       DOI = {10.1006/jctb.2001.2047},
       URL = {https://doi.org/10.1006/jctb.2001.2047},
}

@article{MS_2017,
    AUTHOR = {Song, Wen-Yao and Miao, Lian-Ying},
     TITLE = {Strong edge-coloring of planar graphs},
   JOURNAL = {Discuss. Math. Graph Theory},
  FJOURNAL = {Discussiones Mathematicae. Graph Theory},
    VOLUME = {37},
      YEAR = {2017},
    NUMBER = {4},
     PAGES = {845--857},
      ISSN = {1234-3099,2083-5892},
   MRCLASS = {05C15 (05C10)},
  MRNUMBER = {3684111},
MRREVIEWER = {Deming\ Li},
       DOI = {10.7151/dmgt.1951},
       URL = {https://doi.org/10.7151/dmgt.1951},
}

@article{BI_strong_2013,
    AUTHOR = {Borodin, Oleg V. and Ivanova, Anna O.},
     TITLE = {Precise upper bound for the strong edge chromatic number of
              sparse planar graphs},
   JOURNAL = {Discuss. Math. Graph Theory},
  FJOURNAL = {Discussiones Mathematicae. Graph Theory},
    VOLUME = {33},
      YEAR = {2013},
    NUMBER = {4},
     PAGES = {759--770},
      ISSN = {1234-3099,2083-5892},
   MRCLASS = {05C15 (05C10)},
  MRNUMBER = {3117054},
MRREVIEWER = {Erika\ Feckov\'a{} \v Skrabu\v l\'akov\'a},
       DOI = {10.7151/dmgt.1708},
       URL = {https://doi.org/10.7151/dmgt.1708},
}

@article{Deniz_2024,
    AUTHOR = {Deniz, Zakir},
     TITLE = {On 2-distance 16-coloring of planar graphs with maximum degree
              at most five},
   JOURNAL = {Discrete Math.},
  FJOURNAL = {Discrete Mathematics},
    VOLUME = {348},
      YEAR = {2025},
    NUMBER = {4},
     PAGES = {Paper No. 114379, 13},
      ISSN = {0012-365X,1872-681X},
   MRCLASS = {05C15 (05C10 05C12)},
  MRNUMBER = {4844744},
MRREVIEWER = {Daniele\ Parisse},
       DOI = {10.1016/j.disc.2024.114379},
       URL = {https://doi.org/10.1016/j.disc.2024.114379},
}

@article{Wang_2025,
    AUTHOR = {Wang, Runze},
     TITLE = {Strong edge-coloring of graphs with maximum edge weight seven},
   JOURNAL = {J. Comb. Optim.},
  FJOURNAL = {Journal of Combinatorial Optimization},
    VOLUME = {51},
      YEAR = {2026},
    NUMBER = {1},
     PAGES = {Paper No. 2, 11},
      ISSN = {1382-6905,1573-2886},
   MRCLASS = {05C15},
  MRNUMBER = {5004997},
       DOI = {10.1007/s10878-025-01381-5},
       URL = {https://doi.org/10.1007/s10878-025-01381-5},
}

@article{GS_2023,
    AUTHOR = {Gy\'arf\'as, Andr\'as and S\'ark\"ozy, G\'abor N.},
     TITLE = {``{L}ess'' strong chromatic indices and the {$(7,
              4)$}-conjecture},
   JOURNAL = {Studia Sci. Math. Hungar.},
  FJOURNAL = {Studia Scientiarum Mathematicarum Hungarica. Combinatorics,
              Geometry and Topology (CoGeTo)},
    VOLUME = {60},
      YEAR = {2023},
    NUMBER = {2-3},
     PAGES = {109--122},
      ISSN = {0081-6906,1588-2896},
   MRCLASS = {05C15},
  MRNUMBER = {4705277},
MRREVIEWER = {Hui\ Lei},
       DOI = {10.1556/012.2023.01539},
       URL = {https://doi.org/10.1556/012.2023.01539},
}

@incollection{BES_1973,
    AUTHOR = {Brown, W. G. and Erd{\H o}s, P. and S\'os, V. T.},
     TITLE = {Some extremal problems on {$r$}-graphs},
 BOOKTITLE = {New directions in the theory of graphs ({P}roc. {T}hird {A}nn
              {A}rbor {C}onf., {U}niv. {M}ichigan, {A}nn {A}rbor, {M}ich.,
              1971)},
     PAGES = {53--63},
 PUBLISHER = {Academic Press, New York-London},
      YEAR = {1973},
   MRCLASS = {05C35},
  MRNUMBER = {351888},
MRREVIEWER = {B\'ela\ Bollob\'as},
}

@article{GMRS_2024,
     TITLE = {Proper edge colorings of planar graphs with rainbow C4 ${C}_{4}$‐s},
    VOLUME = {107},
      ISSN = {1097-0118},
       URL = {http://dx.doi.org/10.1002/jgt.23163},
       DOI = {10.1002/jgt.23163},
    NUMBER = {4},
   JOURNAL = {Journal of Graph Theory},
 PUBLISHER = {Wiley},
    AUTHOR = {Gyárfás, András and Martin, Ryan R. and Ruszinkó, Miklós and Sárközy, Gábor N.},
      YEAR = {2024},
     MONTH = Aug, 
     PAGES = {833–846} 
}

@unpublished{MRS_2026,
     TITLE = {Bounds for B-coloring planar and outerplanar graphs},
    AUTHOR = {Martin, Ryan R. and Ruszinkó, Miklós and Sárközy, Gábor N.},
      NOTE = {Preprint}
}

@article{MR_1997,
    AUTHOR = {Molloy, Michael and Reed, Bruce},
     TITLE = {A bound on the strong chromatic index of a graph},
   JOURNAL = {J. Combin. Theory Ser. B},
  FJOURNAL = {Journal of Combinatorial Theory. Series B},
    VOLUME = {69},
      YEAR = {1997},
    NUMBER = {2},
     PAGES = {103--109},
      ISSN = {0095-8956,1096-0902},
   MRCLASS = {05C15},
  MRNUMBER = {1438613},
MRREVIEWER = {Mirko\ Hor\v n\'ak},
       DOI = {10.1006/jctb.1997.1724},
       URL = {https://doi.org/10.1006/jctb.1997.1724},
}

\end{document}